\documentclass[12pt,reqno]{amsart}
\usepackage{amsmath,amsthm,amsfonts,amssymb}
\usepackage{graphicx}

\newtheorem{theorem}{Theorem}[section]
\newtheorem{lemma}[theorem]{Lemma}
\newtheorem{conjecture}[theorem]{Conjecture}
\newtheorem{proposition}[theorem]{Proposition}

\theoremstyle{definition}

\theoremstyle{remark}

\numberwithin{equation}{section}

\begin{document}

\title[Uniform constants in nonlinear Hausdorff-Young inequalities]
{Uniform constants in Hausdorff-Young inequalities for the Cantor group model of the scattering transform}

\author{Vjekoslav Kova\v{c}}
\address{Department of Mathematics, UCLA, Los Angeles, CA 90095-1555}
\email{vjekovac@math.ucla.edu}

\subjclass[2010]{Primary 34L25; Secondary 42A38}

\begin{abstract}
Analogues of Hausdorff-Young inequalities for the Dirac scattering
transform (a.k.a.\@ $SU(1,1)$ nonlinear Fourier transform) were
first established by Christ and Kiselev \cite{CK1},\cite{CK2}. Later Muscalu,
Tao, and Thiele \cite{MTT} raised a question if the constants
can be chosen uniformly in $1\leq p\leq 2$. Here we give a positive answer
to that question when the Euclidean real line is replaced by its
Cantor group model.
\end{abstract}

\maketitle

\section{Introduction}

The following context arises from the eigenfunction equation for the Dirac operator,
after the natural ansatz is made, see \cite{MTT}.
Let $f\colon\mathbb{R}\to\mathbb{C}$ be a compactly supported integrable function.
For any $\xi\in\mathbb{R}$ consider the initial value problem in the matrix form:
\begin{equation}\label{vkivp}
\frac{\partial}{\partial x}G(x,\xi) = G(x,\xi) W(x,\xi),\qquad G(-\infty,\xi)=
\left[\begin{array}{cc} 1 & 0 \\ 0 & 1 \end{array}\right],
\end{equation}
where
$$ G(x,\xi)=\left[\begin{array}{cc} a(x,\xi) & \overline{b(x,\xi)} \\ b(x,\xi) & \overline{a(x,\xi)} \end{array}\right],
\qquad W(x,\xi)=\left[\begin{array}{cc} 0 & \overline{f(x)} e^{-2\pi i x \xi} \\
f(x) e^{2\pi i x \xi} & 0 \end{array}\right]. $$
The problem (\ref{vkivp}) has a unique solution with absolutely continuous functions
$a(\cdot,\xi)$ and $b(\cdot,\xi)$ that satisfy the differential equation for a.e.\@ $x\in\mathbb{R}$ and
eventually become constant as $x\to-\infty$ or $x\to\infty$.
The limit
\begin{equation}\label{vkscatteringdef}
G(\infty,\xi) = \left[\begin{array}{cc} a(\infty,\xi) & \overline{b(\infty,\xi)} \\
b(\infty,\xi) & \overline{a(\infty,\xi)} \end{array}\right]
=\lim_{x\to\infty}\left[\begin{array}{cc} a(x,\xi) & \overline{b(x,\xi)} \\
b(x,\xi) & \overline{a(x,\xi)} \end{array}\right]
\end{equation}
is a function in $\xi\in\mathbb{R}$, called the \emph{Dirac scattering transform} of $f$.
It is easy to see that all matrices $G(x,\xi)$ must belong to the Lie group
$$ \mathrm{SU}(1,1) := \left\{\left[\begin{array}{cc} a & \overline{b} \\ b & \overline{a} \end{array}\right]
\ : \ a,b\in\mathbb{C},\ |a|^2-|b|^2=1 \right\}, $$
and so $\xi\mapsto G(\infty,\xi)$ is indeed a function from $\mathbb{R}$ to $\mathrm{SU}(1,1)$.
In analogy with the (linear) Fourier transform on $\mathbb{R}$, we also call it the $\mathrm{SU}(1,1)$
\emph{nonlinear Fourier transform} of $f$, the term originating in \cite{TT}.
We simply write $G(\xi)$, $a(\xi)$, $b(\xi)$ in place of $G(\infty,\xi)$, $a(\infty,\xi)$, $b(\infty,\xi)$.

Using elementary contour integration one can show a ``nonlinear analogue'' of the Plancherel theorem:
$$ \| (2\ln |a(\xi)|)^{1/2} \|_{\mathrm{L}^2_\xi(\mathbb{R})} = \|f\|_{\mathrm{L}^2(\mathbb{R})} \,. $$
The first appearance of this identity (although in discrete setting) dates back to \cite{V1},\cite{V2}.
From this equality it seems that $(\ln |a|)^{1/2}$ is the appropriate measure of size
for matrices in $\mathrm{SU}(1,1)$, so in the spirit of classical Fourier analysis one can consider
nonlinear analogues of Hausdorff-Young inequalities for $1\leq p<2$:
\begin{equation}\label{vkhy}
\| (\ln |a(\xi)|)^{1/2} \|_{\mathrm{L}^q_\xi(\mathbb{R})} \leq C_p\, \|f\|_{\mathrm{L}^p(\mathbb{R})} ,
\end{equation}
where $p$ and $q$ are conjugated exponents.
Besides the trivial Riemann-Lebesgue type of estimate for $p=1$, one can show (\ref{vkhy})
for $1<p<2$, as is first done in \cite{CK1},\cite{CK2}.
These papers also prove the maximal version of (\ref{vkhy}), i.e.\@ Menshov-Paley-Zygmund type inequality.
Even stronger, variational estimates for $1\leq p<2$ are shown recently in \cite{OSTW}.

However, the truncation method from \cite{CK1},\cite{CK2} gives constants $C_p$ in (\ref{vkhy})
that blow up as $p\to 2-$.
For that reason Muscalu, Tao, and Thiele raised the following conjecture in \cite{MTT}.
\begin{conjecture}\label{vkconjecture}
There exists a universal constant $C>0$ such that for any pair of conjugated exponents
$1\leq p\leq 2$ and $2\leq q\leq \infty$ and every function $f$ as above one has
$$ \| (\ln |a(\xi)|)^{1/2} \|_{\mathrm{L}^q_\xi(\mathbb{R})} \leq C\, \|f\|_{\mathrm{L}^p(\mathbb{R})}. $$
\end{conjecture}
It is interesting to notice that, although we know that (\ref{vkhy}) holds in the endpoint case $p=2$,
we still cannot conclude uniformity of $C_p$ for neighboring values of $p$.
Such anomalies are not possible for linear operators due to the Riesz-Thorin interpolation theorem.
However, our transformation $f\mapsto (\ln |a(\cdot)|)^{1/2}$ is truly nonlinear,
and no standard interpolation result can be applied directly to prove the conjecture.

\smallskip
The goal of this paper is to prove Conjecture \ref{vkconjecture} in the case
when the exponentials $e^{2\pi i x \xi}$ in $W(x,\xi)$ are replaced by
the character function $E_d(x,\xi)$ of the $d$-adic Cantor group model of the real line,
rigorously defined in the next section.
The method of the proof is a monotonicity argument over scales, which is typically
a privilege of finitary group models.
Such arguments are also sometimes called \emph{Bellman function} proofs (see for instance \cite{NT}),
as they require construction of
an auxiliary function with certain monotonicity and convexity properties.

The main idea is taken from the ``local proof'' of the Cantor group model Plancherel theorem given in \cite{MTT}.
A new contribution is the construction of the modified ``swapping function'' $\beta_d$
that satisfies certain $\mathrm{L}^p\to\mathrm{L}^q$ estimates uniformly in $1\leq p\leq 2$.
In the proof we use linear Hausdorff-Young inequalities on $\mathbb{Z}/d\mathbb{Z}$,
as a substitute for some cancellation identities in \cite{MTT}.

\smallskip
Let us remark that our qualitative assumption on $f$ is crucial in order to be able to define the scattering
transform properly.
If $f$ is merely in $\mathrm{L}^p(\mathbb{R})$ for $1\leq p<2$ (but without compact support),
then from maximal inequalities in \cite{CK1},\cite{CK2} it follows that the limit in (\ref{vkscatteringdef})
exists for a.e.\@ $\xi\in\mathbb{R}$, but this is a rather nontrivial result.
However, for $f\in\mathrm{L}^2(\mathbb{R})$ that is still an open problem, commonly known as
the \emph{nonlinear Carleson theorem}.
Its Cantor group model variant is proven in \cite{MTT}.
One can still extend the definition of the scattering transform using density arguments,
as in \cite{TT}.

\smallskip\noindent
\textbf{Acknowledgment.}
The author would like to thank his faculty advisor, Prof.\@ Christoph Thiele,
for suggesting the problem and for his help on improving the presentation.

\section{The monotonicity argument}

Fix an integer $d\geq 2$, and denote $\mathbb{Z}_d:=\mathbb{Z}/d\mathbb{Z}$.
For any $x,\xi\in[0,\infty)$ that can be written uniquely\footnote{
Because of ambiguous base $d$ representation of some reals, the function $E_{d}$ is not well-defined
on a set of measure zero. The same comment applies to the later identification
of $\mathbb{A}_d$ with $[0,\infty)$.}
in base $d$ number system as
$x=\sum_{n\in\mathbb{Z}}x_n d^n$ and $\xi=\sum_{n\in\mathbb{Z}}\xi_n d^n$,
we define
$$ E_{d}(x,\xi) := e^{(2\pi i/d)\sum_{n\in\mathbb{Z}} x_n \xi_{-1-n}} \,. $$
Then the $\mathrm{L}^\infty$ function
$E_{d}\colon[0,\infty)\times[0,\infty)\to\mathrm{S}^1$
is called the \emph{Cantor group character function}.
To justify the name, we identify $[0,\infty)$
with a subgroup $\mathbb{A}_d$ of the infinite group product $\mathbb{Z}_d^\mathbb{Z}$ given by
\begin{align*}
\mathbb{A}_d := \big\{(x_n)_{n\in\mathbb{Z}} \ : \ \ & x_n\in\mathbb{Z}_d\textrm{ for every }n\in\mathbb{Z},
\textrm{ and there exists} \\
& n_0\in\mathbb{Z}\textrm{ such that }x_n=\mathbf{0}\textrm{ for every }n\geq n_0 \big\} \,,
\end{align*}
via the identification
\,$\mathbb{A}_d\to[0,\infty)$,\, $(x_n)_{n\in\mathbb{Z}}\mapsto\sum_{n\in\mathbb{Z}}x_n d^n$.
Then $E_{d}(\cdot,\cdot)$ realizes duality between
$\mathbb{A}_d$ and its dual group $\hat{\mathbb{A}}_d\cong\mathbb{A}_d$.

\smallskip
For a compactly supported integrable function $f\colon[0,\infty)\to\mathbb{C}$ and $\xi\in[0,\infty)$
consider the initial value problem on $[0,\infty)$:
$$ \frac{\partial}{\partial x}G(x,\xi) = G(x,\xi) W(x,\xi),\qquad G(0,\xi)=
\left[\begin{array}{cc} 1 & 0 \\ 0 & 1 \end{array}\right], $$
where
$$ G(x,\xi)=\left[\begin{array}{cc} a(x,\xi) & \overline{b(x,\xi)} \\
b(x,\xi) & \overline{a(x,\xi)} \end{array}\right],
\qquad W(x,\xi)=\left[\begin{array}{cc} 0 & \overline{f(x)\, E_{d}(x,\xi)} \\
f(x)\, E_{d}(x,\xi) & 0 \end{array}\right]. $$
The limit
$$ G(\xi) = \left[\begin{array}{cc} a(\xi) & \overline{b(\xi)} \\
b(\xi) & \overline{a(\xi)} \end{array}\right]
:= \lim_{x\to\infty}\left[\begin{array}{cc} a(x,\xi) & \overline{b(x,\xi)} \\
b(x,\xi) & \overline{a(x,\xi)} \end{array}\right] $$
defines a function $\xi\mapsto G(\xi)$ from $[0,\infty)$ to $\mathrm{SU}(1,1)$,
which we call the \emph{Cantor group model Dirac scattering transform} of $f$.
Dependence on $d$ is not notationally emphasized but is understood.
If for some interval $I\subseteq[0,\infty)$ we replace $f$ by $f\mathbf{1}_I$,
then we will denote the corresponding $G$, $a$, $b$ respectively by
$G_I$, $a_I$, $b_I$.

\smallskip
The main result of the paper is the following theorem.
\begin{theorem}\label{vkmaintheorem}
For every integer $d\geq 2$ there exists a constant $C_d>0$ such that for any pair of
conjugated exponents $1\leq p\leq 2$ and $2\leq q\leq \infty$ and
every function $f$ as above one has
$$ \| (\ln |a(\xi)|)^{1/2} \|_{\mathrm{L}^q_\xi(\mathbb{R})}
\leq C_d \|f\|_{\mathrm{L}^p(\mathbb{R})}. $$
\end{theorem}
The proof is given below, with the main technical construction postponed until the next section.
In the following exposition we need a couple of simple facts proved in \cite{MTT}.
\begin{lemma}[from \cite{MTT}]\label{vklemmamtt1}
If $I$ and $\omega$ are two $d$-adic intervals\footnote{These are intervals of the form
$\big[d^n m,d^n(m+1)\big)$, for some $m,n\in\mathbb{Z}$, $m\geq 0$.}
with $|I| |\omega|=1$, then $\xi\mapsto |a_I(\xi)|$ and $\xi\mapsto |b_I(\xi)|$ are constant functions on $\omega$.
\end{lemma}
We will be working in the phase space $\mathbb{A}_d\times\hat{\mathbb{A}}_d$,
which is identified with $[0,\infty)\times[0,\infty)$.
\emph{Tiles} and \emph{multitiles} are rectangles of the form $I\times\omega$
for two $d$-adic intervals $I$, $\omega$ satisfying
$|I| |\omega|=1$ and $|I| |\omega|=d$ respectively.
Every multitile $I\times\omega$ can be partitioned into $d$ tiles by subdividing
either $I$ or $\omega$ into $d$ congruent $d$-adic intervals.
Lemma \ref{vklemmamtt1} motivates us to define $G_P$, $a_P$, $b_P$ for any tile
$P=I\times\omega$ simply as
$G_I(\xi_\omega)$, $a_I(\xi_\omega)$, $b_I(\xi_\omega)$,
where $\xi_\omega$ is the left endpoint of $\omega$.

\begin{figure}[htbp]
\includegraphics[width=0.45\textwidth]{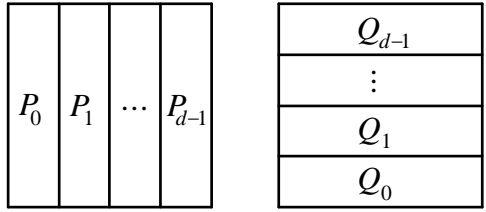}
\caption{A multitile partitioned in two ways.}
\label{vkfiguretiles}
\end{figure}
\begin{lemma}[from \cite{MTT}]\label{vklemmamtt2}
Suppose that a multitile is divided horizontally into tiles $P_0,\ldots,P_{d-1}$,
and vertically into tiles $Q_0,\ldots,Q_{d-1}$, as in Figure \ref{vkfiguretiles}.
Then
$$ \left[\begin{array}{cc} a_{Q_k} & \overline{b_{Q_k}} \\ b_{Q_k} & \overline{a_{Q_k}}\end{array}\right]
= \prod_{j=0}^{d-1} \left[\begin{array}{cc}a_{P_j} & \overline{b_{P_j}} \,e^{-2\pi i j k/d} \\
b_{P_j} \,e^{2\pi i j k/d} & \overline{a_{P_j}}\end{array}\right] $$
for $k=0,1,\ldots,d\!-\!1$.
(The matrix product has to be taken in ascending order.)
\end{lemma}

\smallskip
This section concludes with the proof of Theorem \ref{vkmaintheorem}, assuming that the following
proposition holds.
\begin{proposition}\label{vkmainprop}
There exist a constant $C_d>0$ and a function $\beta_d\colon[0,\infty)\to[0,\infty)$ such that
for every {\scriptsize $\bigg[\!\!\begin{array}{cc} a & \!\!\!\overline{b} \\ b & \!\!\!\overline{a} \end{array}\!\!\bigg]$}
$\in\mathrm{SU}(1,1)$
\begin{equation}\label{vkmainpropeq1}
C_d^{-1} (\ln|a|)^{1/2} \,\leq\, \beta_d(|b|) \,\leq\, C_d (\ln|a|)^{1/2} \,,
\end{equation}
and whenever matrices
{\scriptsize $\bigg[\!\!\begin{array}{cc} a_j & \!\!\!\overline{b_j} \\ b_j & \!\!\!\overline{a_j} \end{array}\!\!\bigg]$},
{\scriptsize $\bigg[\!\!\begin{array}{cc} A_k & \!\!\!\overline{B_k} \\ B_k & \!\!\!\overline{A_k} \end{array}\!\!\bigg]$}
$\in\mathrm{SU}(1,1)$, $j,k=0,1,\ldots,d\!-\!1$ satisfy
\begin{equation}\label{vkmatrices}
\left[\begin{array}{cc} A_k & \overline{B_k} \\ B_k & \overline{A_k}
\end{array}\right]= \prod_{j=0}^{d-1} \left[\begin{array}{cc}
a_j & \overline{b_j} \,e^{-2\pi i j k/d} \\ b_j \,e^{2\pi i j k/d} & \overline{a_j}
\end{array}\right] \,,
\end{equation}
then for any pair of conjugated exponents $1<p\leq 2$ and $2\leq q<\infty$ one has
\begin{equation}\label{vkmainpropeq2}
\Big(\frac{1}{d}\sum_{k=0}^{d-1}\beta_d(|B_k|)^{q}\Big)^{\frac{1}{q}}
\leq \Big(\sum_{j=0}^{d-1}\beta_d(|b_j|)^{p}\Big)^{\frac{1}{p}} \,.
\end{equation}
\end{proposition}
This proposition is proved in the next section, by giving an
explicit construction of $\beta_d$.
The construction might seem a bit tedious, but we have to satisfy (\ref{vkmainpropeq2})
with the exact constant at most $1$, since we will be repeatedly applying that inequality
in the proof of Theorem \ref{vkmaintheorem}.
Iterating an inequality with a constant $C>1$ would not yield an estimate independent of the number of scales.

A consequence of Lemma \ref{vklemmamtt2} and (\ref{vkmainpropeq2}) is that for
$P_0,\ldots,P_{d-1},Q_0,\ldots,Q_{d-1}$ as above we get
\begin{equation}\label{vkmaineq}
\Big(\frac{1}{d}\sum_{k=0}^{d-1}\beta_d(|b_{Q_k}|)^{q}\Big)^{\frac{1}{q}}
\leq \Big(\sum_{j=0}^{d-1}\beta_d(|b_{P_j}|)^{p}\Big)^{\frac{1}{p}}.
\end{equation}

\begin{proof}[Proof of Theorem \ref{vkmaintheorem} assuming Proposition \ref{vkmainprop}]\ \\
We can consider $1<p\leq 2$, as for $p=1$ the estimate is an immediate consequence of
Gronwall's inequality.
Fix a positive integer $N$ (large enough) so that $f$ is supported
in $[0,d^{N})$. In all of the following we consider only those
tiles $I\times\omega$ that are subsets of
$[0,d^{N})\times[0,d^{N})$.
For any $n\in\mathbb{Z}$, $-N\leq n\leq N$ consider the following quantity:
$$ \mathcal{B}_n := \Bigg(\sum_{|I|=d^n}\Big(d^{-n}\!\!\!
\sum_{|\omega|=d^{-n}}\beta_d(|b_{I\times\omega}|)^q\Big)^{\frac{p}{q}}\Bigg)^{\frac{1}{p}}. $$
In words, we consider all tiles $P$ of type $d^{n}\times d^{-n}$,
then we take normalized $\ell^q$-norm of numbers $\beta_d(|b_P|)$ for
all tiles in the same column, and finally we take $\ell^p$-norm of
those numbers over all columns.
Let us first prove that this quantity is decreasing in $n$.
\begin{align*}
\mathcal{B}_{n+1}^p
& = \sum_{|I|=d^{n+1}}\Big(d^{-n}\!\!\!\sum_{|\omega|=d^{-n}}d^{-1}\!\!\!\!\!\!
\sum_{\scriptsize\begin{array}{c}\omega'\subseteq\omega \\|\omega'|=d^{-n-1}\end{array}}
\!\!\!\!\!\beta_d(|b_{I\times\omega'}|)^q\Big)^{\frac{p}{q}} \\
& \textrm{using (\ref{vkmaineq}) for the multitile $I\times\omega$} \\
& \leq \sum_{|I|=d^{n+1}}\bigg(d^{-n}\!\!\!\sum_{|\omega|=d^{-n}}
\Big(\!\!\sum_{\scriptsize\begin{array}{c}I'\subseteq I\\
|I'|=d^{n}\end{array}}
\beta_d(|b_{I'\times\omega}|)^p\Big)^{\frac{q}{p}}\bigg)^{\frac{p}{q}} \\
& \textrm{using Minkowski's inequality,\, since $q/p\geq 1$} \\
& \leq \sum_{|I|=d^{n+1}}\sum_{\scriptsize\begin{array}{c}I'\subseteq I\\|I'|=d^{n}\end{array}}
\Big(d^{-n}\!\!\!\sum_{|\omega|=d^{-n}}\beta_d(|b_{I'\times\omega}|)^q\Big)^{\frac{p}{q}}
= \mathcal{B}_n^p
\end{align*}
Furthermore, when $n=-N$ we have:
\begin{align*}
\mathcal{B}_{-{N}} = \Big(\sum_{|I|=d^{-N}}(d^{N})^{\frac{p}{q}}\,\,\beta_d(|b_{I\times[0,d^N)}|)^p\Big)^{\frac{1}{p}}
\leq C_d (d^{N})^{\frac{1}{q}} \Big(\sum_{|I|=d^{-{N}}} (\ln|a_I(0)|)^{\frac{p}{2}}\Big)^{\frac{1}{p}} & \\
\leq C_d (d^{N})^{\frac{1}{q}} \Big(\sum_{|I|=d^{-{N}}} \|f\mathbf{1}_I\|^p_{\mathrm{L}^1} \Big)^{\frac{1}{p}}
\leq C_d \|f\|_{\mathrm{L}^p} & \,.
\end{align*}
Here we have applied the trivial $\mathrm{L}^1$--$\mathrm{L}^\infty$ estimate
$(\ln|a_I(0)|)^{1/2}\leq\|f\mathbf{1}_I\|_{\mathrm{L}^1}$ and H\"{o}lder's inequality
$\|f\mathbf{1}_I\|_{\mathrm{L}^1}\leq\|f\mathbf{1}_I\|_{\mathrm{L}^p}\|\mathbf{1}_I\|_{\mathrm{L}^q}$.
On the other hand, for $n=N$ we have:
\begin{align*}
\mathcal{B}_{N} = \Big(d^{-N}\!\!\sum_{|\omega|=d^{-N}}\beta_d(|b_{[0,d^{N})\times\omega}|)^q\Big)^{\frac{1}{q}}
\geq C_d^{-1} \Big(d^{-N}\!\!\sum_{|\omega|=d^{-{N}}}(\ln|a(\xi_\omega)|)^\frac{q}{2}\Big)^{\frac{1}{q}} & \\
= C_d^{-1}\Big(\int_{0}^{d^{N}}\!\!\!(\ln|a(\xi)|)^\frac{q}{2}\,d\xi\Big)^{\frac{1}{q}} & \,.
\end{align*}
Above $\xi_\omega$ denotes the left endpoint of $\omega$ and we
have used that $\xi\mapsto |a(\xi)|$ is constant on intervals
of length $d^{-N}$, by Lemma \ref{vklemmamtt1}.
From the monotonicity of $(\mathcal{B}_n)$ we conclude:
$$ \Big(\int_{0}^{d^{N}}\!\!\!(\ln|a(\xi)|)^\frac{q}{2} \,d\xi\Big)^{\frac{1}{q}}
\leq C_d\mathcal{B}_{N} \leq C_d\mathcal{B}_{-N} \leq C_d^2 \|f\|_{\mathrm{L}^p} \,, $$
and by taking $\lim_{N\to\infty}$ we deduce the theorem.
\end{proof}

\section{The swapping inequality}

This technical section is devoted to the proof of Proposition \ref{vkmainprop}.
An arbitrary function on $\mathbb{Z}_d$ can be presented as a complex
$d$-tuple $(z_0,z_1,\ldots,z_{d-1})$. Its Fourier transform is the
$d$-tuple $(Z_0,Z_1,\ldots,Z_{d-1})$ given by
$$ Z_k := \sum_{j=0}^{d-1} z_j \,e^{2\pi i j k/d} \,. $$
\begin{lemma}\label{vklemmahy}
For a pair of conjugated exponents $1< p\leq 2$ and $2\leq q< \infty$
and $(z_j)$, $(Z_k)$ as above, one has
$$ \Big(\frac{1}{d}\sum_{k=0}^{d-1}|Z_k|^q\Big)^{\frac{1}{q}} \leq \Big(\sum_{j=0}^{d-1}|z_j|^p\Big)^{\frac{1}{p}} . $$
\end{lemma}
Lemma \ref{vklemmahy} is a particular consequence of the general theory of the Fourier transform
on locally compact abelian groups (see \cite{F}).
Indeed, one observes that the (non-stated) case $p=1$ is trivial from the triangle inequality,
while for $p=2$ we indeed have an equality that follows from orthonormality of group characters.
Intermediate cases are deduced by interpolating these two endpoint ones using the Riesz-Thorin theorem,
since the transformation $(z_j)\mapsto(Z_k)$ is linear.

\smallskip
For any integer $d\geq 2$\, let $t_d$ be the unique solution of
the equation
$$ t e^{-t} = (2d)^{-5}\sqrt{1 + \mathop{\mathrm{arsinh}} t}  $$
that lies in $[0,1]$. One can easily see
\begin{equation}\label{vkboundsfortn}
2^{-5}d^{-5} < t_d < 2^{-4}d^{-5} \,,
\end{equation}
and indeed \ $t_d = (2d)^{-5} + \frac{3}{2}(2d)^{-10} +
\mathrm{O}(d^{-15})$ \ as $d\to\infty$, \ but we do not need
bounds on $t_d$ that are more precise than (\ref{vkboundsfortn}).

Now we define $\beta_d\colon[0,\infty)\to[0,\infty)$ by the formula
$$ \beta_d(t):=\left\{\begin{array}{cl}
t e^{-t}, & \textrm{for } t\leq t_d \,, \\[2mm]
(2d)^{-5}\sqrt{1 + \mathop{\mathrm{arsinh}} t}, &  \textrm{for } t> t_d \,.
\end{array}\right. $$
\begin{figure}[htbp]
\includegraphics[width=0.75\textwidth]{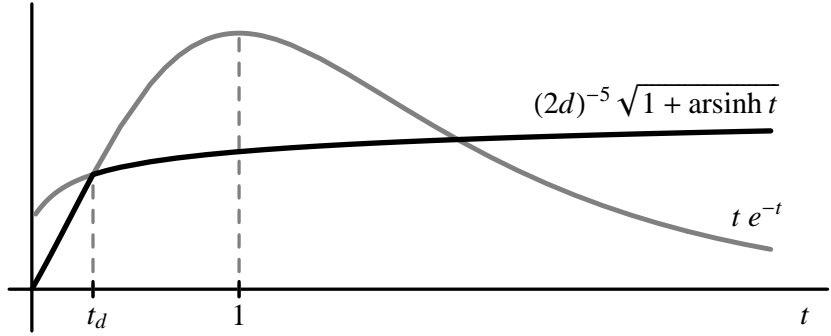}
\caption{Graph of $\beta_d$. Not drawn to scale.}
\label{vkfigurebeta}
\end{figure}

\smallskip\noindent
Using only basic calculus, one can easily establish the following properties of $\beta_d$:
\begin{align}
\label{vkbetaineq}
& 2^{-6} d^{-5} \sqrt{\ln(1+t^2)} \ \, \leq \ \beta_d(t) \ \leq \ 2\, \sqrt{\ln(1+t^2)} \,, \\
\label{vkbetaineq2}
& \beta_d(t) \ \leq \ t e^{-t}, \qquad \textrm{for }0\leq t\leq 1 \,, \\
\label{vkbetaineq3}
& \beta_d(t) \ \leq \ (2d)^{-5} \,\sqrt{1 + \mathop{\mathrm{arsinh}} t}, \qquad \textrm{for any }t\geq 0 \,.
\end{align}
Since (\ref{vkbetaineq}) is exactly (\ref{vkmainpropeq1}),
it is enough to verify (\ref{vkmainpropeq2}).

\smallskip
By performing matrix multiplication in (\ref{vkmatrices}), one can
write $B_k$ explicitly as a sum of $2^{d-1}$ terms of the form
\begin{align*}
& \overline{a_{0}}\ldots\overline{a_{j_1-1}}b_{j_1}a_{j_1+1}\ldots a_{j_2-1}\overline{b_{j_2}}
\overline{a_{j_2+1}}\ldots\overline{a_{j_3-1}}b_{j_3}a_{j_3+1}\ldots \\
& \ldots a_{j_{2r}-1}\overline{b_{j_{2r}}}
\overline{a_{j_{2r}+1}}\ldots\overline{a_{j_{2r+1}-1}}b_{j_{2r+1}}a_{j_{2r+1}+1}\ldots a_{d-1}
\cdot e^{(2\pi i k/d)(j_1-j_2+j_3-\ldots-j_{2r}+j_{2r+1})} &
\end{align*}
where the summation is taken over all integers $0\leq
r\leq\lfloor\frac{d-1}{2}\rfloor$ and over all possible choices of
indices $0\leq j_1<j_2<\ldots<j_{2r+1}\leq d\!-\!1$. In
particular, observe that each term contains an odd number of
$b$'s. Terms that contain exactly one of the $b$'s could be called
\emph{linear terms}, and so the ``linear part'' of $B_k$ is
$$ B'_k :=  \sum_{j=0}^{d-1} \overline{a_0 a_1 \ldots a_{j-1}} b_j a_{j+1} \ldots a_{d-1} \,e^{2\pi i j k/d},
\quad\textrm{for } k=0,1,\ldots,d\!-\!1 . $$
Other terms in $B_k$ are called \emph{nonlinear terms}. Observe that Lemma \ref{vklemmahy} gives
\begin{equation}\label{vklineartermshy}
\Big(\frac{1}{d}\sum_{k=0}^{d-1}|B'_k|^q\Big)^{\frac{1}{q}} \leq
\Big(\sum_{j=0}^{d-1}|b'_j|^p\Big)^{\frac{1}{p}} \,,
\end{equation}
where $b'_j := \overline{a_0 \ldots a_{j-1}} b_j a_{j+1} \ldots a_{d-1}$.
In the case when some $|b_m|$ is ``large'' and all other $|b_j|$, $j\neq m$ are ``small'' we find the following
variant more useful:
\begin{align*}
B''_k \ := & \ \sum_{j=0}^{m-1} \overline{c_0\ldots c_{j-1}} b_j c_{j+1}\ldots
c_{m-1} a_m c_{m+1}\ldots c_{d-1} \,e^{2\pi i j k/d} \\
+ & \ \overline{c_0 \ldots c_{m-1}} b_m c_{m+1} \ldots c_{d-1} \,e^{2\pi i m k/d} \\
+ & \sum_{j=m+1}^{d-1} \overline{c_0\ldots c_{m-1} a_m
c_{m+1}\ldots c_{j-1}} b_j c_{j+1}\ldots c_{d-1} \,e^{2\pi i j k/d} \,,
\end{align*}
where we have denoted $c_j:=a_j/|a_j|$. This time Lemma \ref{vklemmahy} implies
\begin{equation}\label{vklineartermshyvar2}
\Big(\frac{1}{d}\sum_{k=0}^{d-1}|B''_k|^q\Big)^{\frac{1}{q}} \
\leq \ \Big(|b_m|^p + |a_m|^p\sum_{j\neq m}|b_j|^p\Big)^{\frac{1}{p}} .
\end{equation}

\smallskip
The proof strategy is to compare $B_k$ to $B'_k$ or $B''_k$ by estimating nonlinear terms,
and then use inequalities (\ref{vklineartermshy}) or (\ref{vklineartermshyvar2}).
As we will soon see, $\beta_d$ is carefully chosen so that it compensates for the perturbation caused by nonlinear terms.

Choose indices $m,m^\ast\in\{0,\ldots,d\!-\!1\}$ such that $|b_m|$
is the largest among the numbers $|b_j|$, and $|b_{m^\ast}|$ is
the largest among the numbers $|b_j|$; $j\neq m$, i.e.\@ the
second largest among $|b_j|$. We distinguish the following three cases.

\medskip
\noindent\emph{Case 1.} \ $|b_j|\leq t_d$ for every $j$.

\smallskip
Recall that $|a_j|^2-|b_j|^2=1$, which implies $|a_j|\leq
1+|b_j|\leq 1+t_d$. We begin with a rough estimate obtained using (\ref{vkboundsfortn}):
$$ |B_k| \ \leq \ \sum_{j=0}^{d-1} |b_j| \Big(\prod_{l\neq j}(|a_l|+|b_l|)\Big)
\ \leq \ d t_d (1+2t_d)^{d-1} \,\leq\, 2^{-3}d^{-4}, $$ which
guarantees $|B_k|\leq 1$, and thus
\,$\beta_d(|B_k|)\leq |B_k| e^{-|B_k|}$ by (\ref{vkbetaineq2}).
Therefore it is enough to prove
$$ \Big(\frac{1}{d}\sum_{k=0}^{d-1}|B_k|^{q} e^{-q|B_k|}\Big)^{\frac{1}{q}}
\leq \Big(\sum_{j=0}^{d-1}|b_j|^{p} e^{-p|b_j|}\Big)^{\frac{1}{p}} \,. $$

\begin{lemma}\label{vkaux1}
\begin{align}
\label{vkaux1eq1} & \Big(\frac{1}{d}\sum_{k=0}^{d-1}|B_k|^{q}
e^{-q|B_k|}\Big)^{\frac{1}{q}} \leq
\Big(\frac{1}{d}\sum_{k=0}^{d-1}|B'_k|^{q}
e^{-q|B'_k|}\Big)^{\frac{1}{q}}
+ 2^{-3}d^{-2} |b_{m^\ast}|^2 \\
\label{vkaux1eq2} & \Big(\frac{1}{d}\sum_{k=0}^{d-1}|B'_k|^{q}
e^{-q|B'_k|}\Big)^{\frac{1}{q}}
\leq \|b'\|_{\ell^p} \, e^{-\|b'\|_{\ell^p}} \\
\label{vkaux1eq3} &
\|b'\|_{\ell^p} \, e^{-\|b'\|_{\ell^p}} \leq
\Big(\sum_{j=0}^{d-1}|b_j|^{p} e^{-p|b_j|}\Big)^{\frac{1}{p}} -
2^{-3}d^{-2} |b_{m^\ast}|^2
\end{align}
Here we have denoted
\ $\|b'\|_{\ell^p} := \big(\sum_{j=0}^{d-1}|b'_j|^{p}\big)^{1/p}$.
\end{lemma}
The desired inequality is obtained simply by adding the three estimates above.

\begin{proof}[Proof of Lemma \ref{vkaux1}]
We start by showing (\ref{vkaux1eq1}).
Since $B_k-B'_k$ contains only nonlinear terms and these have at least $3$ $b$'s,
we have the following error estimate:
\begin{align*}
|B_k - B'_k| & \leq \sum_{j_1<j_2<j_3} |b_{j_1}| |b_{j_2}| |b_{j_3}| \Big(\prod_{l\neq j_1,j_2,j_3}(|a_l|+|b_l|)\Big) \\
& \leq d^3 |b_m| |b_{m^\ast}|^2 (1+2t_d)^{d-3} \leq
2^{-3}d^{-2}|b_{m^\ast}|^2 .
\end{align*}
(For $d=2$ this difference is $0$.) By the mean value theorem for $t e^{-t}$:
$$ \Big| |B_k| e^{-|B_k|}-|B'_k| e^{-|B'_k|} \Big|\leq  2^{-3}d^{-2}|b_{m^\ast}|^2 , $$
and it remains to use Minkowski's inequality.

In order to prove (\ref{vkaux1eq2}) we consider the function
$\varphi(t):= t e^{-q t^{1/q}}$,
which is increasing and concave on $[0,1]$ since:
$$ \varphi'(t)=e^{-q t^{1/q}}(1-t^{1/q})>0, \quad
\varphi''(t)=\frac{1}{q}e^{-q t^{1/q}}t^{1/q-1}(-1-q+q t^{1/q})<0, $$
for $0<t<1$.
Now (\ref{vkaux1eq2}) follows using Jensen's inequality and (\ref{vklineartermshy}):
$$ \frac{1}{d}\sum_{k=0}^{d-1}\varphi(|B'_k|^q) \leq
\varphi\Big(\frac{1}{d}\sum_{k=0}^{d-1}|B'_k|^q\Big) \leq
\varphi\bigg(\Big(\sum_{j=0}^{d-1}|b'_j|^p\Big)^{\frac{q}{p}}\bigg) \,. $$

To show (\ref{vkaux1eq3}), we observe that $|b'_j|\geq |b_j|$,
and thus it suffices to prove
\begin{equation}\label{vkaux1eq3ver2}
\Big(\sum_{j=0}^{d-1}|b_j|^{p} e^{-p|b_j|}\Big)^{\frac{1}{p}}-
\Big(\sum_{j=0}^{d-1}|b'_j|^p
e^{-p\,\|b\|_{\ell^p}}\Big)^{\frac{1}{p}} \geq
2^{-3}d^{-2} |b_{m^\ast}|^2.
\end{equation}
From the mean value theorem we obtain
$$ e^{-p|b_j|} - e^{-p\,\|b\|_{\ell^p}} \geq
p \,e^{-p\,\|b\|_{\ell^p}}
\big(\|b\|_{\ell^p}-|b_j|\big) \,, $$
and using \,$e^{-p\,\|b\|_{\ell^p}} \geq e^{-pd|b_m|}\geq \frac{1}{2}$\, we come to the inequality
\begin{equation}\label{vkaux1eq3add}
|b_j|^{p} e^{-p|b_j|} - |b'_j|^p e^{-p\,\|b\|_{\ell^p}} \geq \ \frac{1}{2}|b_j|^{p}
\,\Big(p\,\|b\|_{\ell^p}-p\,|b_j|-\prod_{l\neq j}|a_l|^p +1\Big).
\end{equation}
Another application of the mean value theorem, this time for the function $t^{1/p}$, gives
$$ p\,\|b\|_{\ell^p}-p|b_j| \ \geq \
d^{-1} |b_m|^{1-p}\Big(\sum_{l\neq j}|b_l|^p \Big) \,. $$
On the other hand, we estimate:
\begin{align*}
& \prod_{l\neq j}|a_l|^p -1 \leq \prod_{l\neq j}|a_l|^2 -1 = \prod_{l\neq j}(1+|b_l|^2) -1
\leq e^{\sum_{l\neq j}|b_l|^2} - 1 \\
& \leq\ e^{2^{-8} d^{-9}} \sum_{l\neq j}|b_l|^2 \leq 2
\Big(\sum_{l\neq j}|b_l|^p\Big)^{\frac{2}{p}}
\leq 2d\,|b_m|^{2-p} \Big(\sum_{l\neq j}|b_l|^p\Big) \,,
\end{align*}
to conclude for every $j\neq m$:
$$ p\,\|b\|_{\ell^p}-p\,|b_j|-\prod_{l\neq j}|a_l|^p +1 \ \geq \ 0, $$
and for $j=m$:
$$ p\,\|b\|_{\ell^p}-p\,|b_m|-\prod_{l\neq m}|a_l|^p +1
\ \geq\ 2^{-1}d^{-1} |b_m|^{1-p} |b_{m^\ast}|^p \,. $$
Now by summing
(\ref{vkaux1eq3add}) over all $j=0,\ldots,d\!-\!1$ we get
$$ \sum_{j=0}^{d-1}|b_j|^{p} e^{-p|b_j|} - \sum_{j=0}^{d-1}|b'_j|^p e^{-p\,\|b\|_{\ell^p}}
\geq 2^{-2}d^{-1}|b_m| |b_{m^\ast}|^p , $$
and then finally obtain (using the mean value theorem for $t^{1/p}$):
\begin{align*}
& \Big(\sum_{j=0}^{d-1}|b_j|^{p} e^{-p|b_j|}\Big)^{\frac{1}{p}} -
\Big(\sum_{j=0}^{d-1}|b'_j|^p e^{-p\,\|b\|_{\ell^p}}\Big)^{\frac{1}{p}} \\
& \geq \frac{1}{p}(d|b_m|^p)^{\frac{1}{p}-1} 2^{-2}d^{-1}|b_m|
|b_{m^\ast}|^p \geq 2^{-3} d^{-2} |b_m|^{2-p} |b_{m^\ast}|^p \geq
2^{-3} d^{-2} |b_{m^\ast}|^2 \,.
\end{align*}
This is exactly (\ref{vkaux1eq3ver2}), which completes the proof of Lemma \ref{vkaux1}.
\end{proof}

\medskip
\noindent\emph{Case 2.} \ $|b_m|>t_d$, but $|b_j|\leq t_d$ for
every $j\neq m$.

\smallskip
By (\ref{vkbetaineq3}) it is enough to prove
\begin{align*}
& (2d)^{-5} \Big(\frac{1}{d}\sum_{k=0}^{d-1}(1+\mathop{\mathrm{arsinh}}|B_k|)^{q/2}\Big)^{\frac{1}{q}} \\[-1mm]
& \leq \Big( (2d)^{-5p} (1+\mathop{\mathrm{arsinh}}|b_m|)^{p/2} +
\sum_{j\neq m} |b_j|^p e^{-p|b_j|} \Big)^\frac{1}{p} ,
\end{align*}
and because \ $e^{-|b_{m^\ast}|}\geq e^{-t_d}\geq\frac{1}{2}$, \
it suffices to show
\begin{equation}\label{vkcase2reduced}
\Big(\frac{1}{d}\sum_{k=0}^{d-1}(1+\mathop{\mathrm{arsinh}}|B_k|)^{q/2}\Big)^{\frac{p}{q}}
\leq (1+\mathop{\mathrm{arsinh}}|b_m|)^{p/2} + 2^{4p} d^{5p}
|b_{m^\ast}|^p \,.
\end{equation}

\begin{lemma}\label{vkaux2}
\begin{align}
\label{vkaux2eq1} &
\Big(\frac{1}{d}\sum_{k=0}^{d-1}(1+\mathop{\mathrm{arsinh}}|B_k|)^{q/2}\Big)^{\frac{p}{q}}
\leq \bigg(1+\mathop{\mathrm{arsinh}}\Big(\frac{1}{d}\sum_{k=0}^{d-1}|B_k|^q\Big)^{\frac{1}{q}}\bigg)^{\frac{p}{2}} \\
\nonumber &
\bigg(1+\mathop{\mathrm{arsinh}}\Big(\displaystyle\frac{1}{d}\sum_{k=0}^{d-1}|B_k|^q\Big)^{\frac{1}{q}}\bigg)^{\frac{p}{2}}
\leq \ \big(1+\mathop{\mathrm{arsinh}}|b_m|\big)^{p/2} \\[-3.5mm]
\label{vkaux2eq2} &
\qquad\qquad\qquad\qquad\qquad\qquad\quad
+ \frac{1}{|a_m|}\bigg(\Big(\displaystyle\frac{1}{d}\sum_{k=0}^{d-1}|B_k|^q\Big)^{\frac{1}{q}}-|b_m|\bigg) \\
\label{vkaux2eq3} &
\Big(\frac{1}{d}\sum_{k=0}^{d-1}|B_k|^q\Big)^{\frac{1}{q}} \leq
\Big(|b_m|^p \!+\! |a_m|^p\!\sum_{j\neq m}|b_j|^p\Big)^{\frac{1}{p}}
+ 2^{4p-1}d^{5p} |a_m| |b_{m^\ast}|^p \\
\label{vkaux2eq4} & \Big(|b_m|^p \!+\! |a_m|^p\!\sum_{j\neq m}|b_j|^p\Big)^{\frac{1}{p}} \leq |b_m| + 2^{4p-1}d^{5p} |a_m|
|b_{m^\ast}|^p
\end{align}
\end{lemma}
Estimate (\ref{vkcase2reduced}) follows by successively substituting
left hand side of each inequality (\ref{vkaux2eq2})--(\ref{vkaux2eq4}) into the preceding one.
Also, we may assume
\ $(\frac{1}{d}\sum_{k}|B_k|^q )^{1/q}$ $\geq |b_m|$ \
in (\ref{vkaux2eq2}),
since otherwise the desired estimate (\ref{vkcase2reduced}) trivially follows from (\ref{vkaux2eq1}).

\begin{proof}[Proof of Lemma \ref{vkaux2}]

In order to prove (\ref{vkaux2eq1}), we consider the function
$$ \psi(t):=\Big(1+\mathop{\mathrm{arsinh}}(t^{2/q})\Big)^\frac{q}{2} . $$
One can calculate:
\begin{align*}
& \psi'(t) = t^{2/q-1}(1+t^{4/q})^{-\frac{1}{2}}
\Big(1+\mathop{\mathrm{arsinh}}(t^{2/q})\Big)^{\frac{q}{2}-1} , \\
& \psi''(t) = -\frac{1}{2q} t^{2/q-2}(1+t^{4/q})^{-\frac{3}{2}}
\Big(1+\mathop{\mathrm{arsinh}}(t^{2/q})\Big)^{\frac{q}{2}-2} \\
& \cdot\Big( 2\big((q\!-\!2)\!+\!q\, t^{4/q}\big)
\mathop{\mathrm{arsinh}}(t^{2/q}) +
(q\!-\!2)\big((1\!+\!t^{4/q})^{\frac{1}{2}}\!-\!t^{2/q}\big)^2
+ (q\!-\!2) + 4t^{4/q} \Big) ,
\end{align*}
and conclude (using $q\geq 2$) that $\psi$ is increasing and concave on $[0,\infty)$.
Jensen's inequality and elementary inequalities between power means (see \cite{MV}) give (\ref{vkaux2eq1}):
$$ \frac{1}{d}\sum_{k=0}^{d-1}\psi(|B_k|^{q/2})
\leq \psi\Big(\frac{1}{d}\sum_{k=0}^{d-1}|B_k|^{q/2}\Big) \leq
\psi\bigg(\Big(\frac{1}{d}\sum_{k=0}^{d-1}|B_k|^q\Big)^\frac{1}{2}\bigg) \,. $$

A couple of applications of the mean value theorem,
for $(1+t)^{p/2}$ and for $\mathop{\mathrm{arsinh}}t$,
together with $1\leq p\leq 2$, $\sqrt{1+|b_m|^2} = |a_m|$, yield (\ref{vkaux2eq2}).

For (\ref{vkaux2eq3}) we first estimate the perturbation due to nonlinear terms:
\begin{align*}
|B_k - B'_k| & \leq \sum_{j_1<j_2<j_3} |b_{j_1}| |b_{j_2}| |b_{j_3}| \Big(\prod_{l\neq j_1,j_2,j_3}(|a_l|+|b_l|)\Big) \\
& \leq d^3 (1+2t_d)^{d-3} (|a_m|+|b_m|) |b_{m^\ast}|^2 \leq 4 d^3
|a_m| |b_{m^\ast}|^2 \,,
\end{align*}
and furthermore compare:
\begin{align*}
|B'_k - B''_k| \ & \leq |b_m|\Big(\prod_{l\neq m}|a_l| -1\Big)+\sum_{j\neq m}|a_m| |b_j|\Big(\prod_{l\neq m,j}|a_l| -1\Big) \\
& \leq d |a_m| \Big(\prod_{l\neq m}|a_l| -1\Big) \leq d |a_m| \Big(e^{\frac{1}{2}\sum_{l\neq m}\!|b_l|^2} -1\Big)
\leq d^2 |a_m| |b_{m^\ast}|^2 \,,
\end{align*}
where in the last line we used $e^x-1 \leq e^x x$ for $x\geq 0$.
These two estimates can be combined,
so that Minkowski's inequality,
together with (\ref{vklineartermshyvar2}),\, $5d^3\leq 2^{4p-1}d^{5p}$, and
$|b_{m^\ast}|^2 \leq |b_{m^\ast}|^p$, gives (\ref{vkaux2eq3}).

To deduce the last estimate (\ref{vkaux2eq4}), we use the mean value theorem
for $t^{1/p}$, and
$\frac{|a_m|}{|b_m|}\leq 1+\frac{1}{|b_m|}\leq\frac{2}{t_d}\leq 2^6 d^5$.
\begin{align*}
& \Big(|b_m|^p \!+\! |a_m|^p\!\sum_{j\neq m}|b_j|^p\Big)^{\frac{1}{p}} - |b_m|
\ \leq \ (|b_m|^p)^{\frac{1}{p}-1} |a_m|^p\!\sum_{j\neq m}|b_j|^p \\
& \leq \ \frac{d |a_m|^p |b_{m^\ast}|^p}{|b_m|^{p-1}} \ \leq \
2^{6p-6} d^{5p-4} |a_m| |b_{m^\ast}|^p \leq \ 2^{4p-1} d^{5p}
|a_m| |b_{m^\ast}|^p \,.
\end{align*}
This proves Lemma \ref{vkaux2}.
\end{proof}

\medskip
\noindent\emph{Case 3.} \ $|b_m|>t_d$ and $|b_{m^\ast}|>t_d$.

\smallskip
Observe that it suffices to prove
$$ \beta_d(|B_k|)^2\leq \sum_{j=0}^{d-1}\beta_d(|b_j|)^{2}, $$
for every $k=0,\ldots,d-1$, because then by elementary
inequalities for $\ell^p$ norms (see \cite{MV}) we have
$$ \Big(\frac{1}{d}\sum_{k=0}^{d-1}\beta_d(|B_k|)^{q}\Big)^{\frac{1}{q}}
\leq \max_{0\leq k\leq d-1}\beta_d(|B_k|) \leq
\Big(\sum_{j=0}^{d-1}\beta_d(|b_j|)^{2}\Big)^{\frac{1}{2}} \leq
\Big(\sum_{j=0}^{d-1}\beta_d(|b_j|)^{p}\Big)^{\frac{1}{p}}. $$

The rest of the proof is a simple observation taken for instance
from \cite{MTT} or \cite{TT}, but we include it for completeness.
Split the set of indices $\{0,\ldots,d\!-\!1\}$ into
$$ J_{\mathrm{big}} := \{j \,:\, |b_j|>t_d \} \qquad\textrm{and}\qquad
J_{\mathrm{small}} := \{j \,:\, |b_j|\leq t_d \} , $$ so in this
case $|J_{\mathrm{big}}|\geq 2$.
Since the spectral norm of any
{\scriptsize $\bigg[\!\!\begin{array}{cc} a & \!\!\!\overline{b} \\ b & \!\!\!\overline{a} \end{array}\!\!\bigg]$}
$\in\mathrm{SU}(1,1)$ is equal to $|a|+|b|$,
using submultiplicativity of operator norms and (\ref{vkmatrices}) we deduce
$$ |A_k|+|B_k| \leq \prod_{j=0}^{d-1} (|a_j| + |b_j|), $$
which can, after taking logarithms, be written as
$$ \mathop{\mathrm{arsinh}}|B_k| \leq \sum_{j=0}^{d-1} \mathop{\mathrm{arsinh}}|b_j|. $$
By (\ref{vkbetaineq3}) one always has
$$ \beta_d(|B_k|)^2 \leq (2d)^{-10}(1+\mathop{\mathrm{arsinh}}|B_k|)\,, $$
and thus we estimate:
\begin{align*}
\beta_d(|B_k|)^2 \, & \leq (2d)^{-10}\Big(1+\sum_{j\in
J_{\mathrm{big}}}\!\mathop{\mathrm{arsinh}}|b_j|
+|J_{\mathrm{small}}|\mathop{\mathrm{arsinh}} t_d\Big) \\[-2mm]
& \leq (2d)^{-10}\!\!\sum_{j\in J_{\mathrm{big}}}(1+\mathop{\mathrm{arsinh}}|b_j|)
= \sum_{j\in J_{\mathrm{big}}}\!\!\beta_d(|b_j|)^{2} \leq\sum_{j=0}^{d-1}\beta_d(|b_j|)^{2} \,.
\end{align*}
In the above calculation we used \ $d \mathop{\mathrm{arsinh}} t_d
\leq d t_d \leq 1$ \ and $|J_{\mathrm{big}}|\geq 2$.

This concludes the last case, and therefore Proposition \ref{vkmainprop} is established.

\section{A closing remark}

While the estimate of Theorem \ref{vkmaintheorem} is independent
of $p$, the proof makes it seriously dependent on $d$.
It is not clear if the latter dependence can
be avoided, but if so, it would require genuinely new methods.
Suppose for a moment that we can construct $\beta=\beta_d$ as in
Proposition \ref{vkmainprop} that does not depend on $d$.
If we take $d\to\infty$ in (\ref{vkmainpropeq2}), we will recover an
analogue of Conjecture \ref{vkconjecture} on the group
$\mathbb{Z}$ (as stated in \cite{TT}), and by an easy transference
principle the actual Conjecture \ref{vkconjecture} (on $\mathbb{R}$) would also follow.
Therefore, uniformization of the constants in $d$
turns out to be an even harder problem than the original one.

\bibliographystyle{amsplain}

\end{document}